\documentclass[12pt]{amsart}

\usepackage{amssymb}

\usepackage{enumerate}

\usepackage{amscd}
\usepackage{mathrsfs}

\usepackage{graphicx}

\makeatletter
\@namedef{subjclassname@2010}{%
  \textup{2010} Mathematics Subject Classification}
\makeatother

\newtheorem{thm}{Theorem}[section]

\newtheorem{lem}[thm]{Lemma}
\newtheorem{prop}[thm]{Proposition}

\newtheorem{conj}[thm]{Conjecture}

\theoremstyle{definition}
\newtheorem{defin}[thm]{Definition}
\newtheorem{rem}[thm]{Remark}

\numberwithin{equation}{section}

\begin{document}


\baselineskip=17pt



\title[volume and Harder-Narasimhan filtration]{On the volume of a pseudo-effective class and  semi-positive properties of the Harder-Narasimhan filtration on a compact  Hermitian manifold}
\author[Z. Wang]{Zhiwei Wang}
\address{ School
of Mathematical Sciences\\Peking University\\ Beijing 100871, China}
\email{wangzw@amss.ac.cn}

\date{}

\begin{abstract}
This paper divides into two parts.  Let $(X,\omega)$ be a compact Hermitian manifold. Firstly, if the Hermitian metric $\omega$ satisfies the assumption that $\partial\overline{\partial}\omega^k=0$ for all $k$, we generalize the volume of the cohomology class in the K\"{a}hler setting to the Hermitian setting, and prove that the volume is always finite  and the Grauert-Riemenschneider type criterion holds true, which is a partial answer to a conjecture posed by Boucksom. Secondly, we observe that  if the anticanonical bundle $K^{-1}_X$ is nef, then for any $\varepsilon>0$, there is a smooth function $\phi_\varepsilon$ on $X$ such that $\omega_\varepsilon:=\omega+i\partial\overline{\partial}\phi_\varepsilon>0$ and  Ricci$(\omega_\varepsilon)\geq-\varepsilon\omega_\varepsilon$. Furthermore, if $\omega$ satisfies the assumption as above, we prove that for a Harder-Narasimhan filtration of $T_X$ with respect to $\omega$, the slopes $\mu_\omega(\mathcal{F}_i/\mathcal{F}_{i-1})\geq 0$ for all $i$, which generalizes a result of Cao which plays a very important role in his studying of the structures of K\"{a}hler manifolds.
\end{abstract}



\maketitle

\section{Introduction}
In this paper, we recall some results in K\"{a}hler geometry  and study to what extend they can be generalized to the case of Hermitian manifolds.

Let $L$ be a holomorphic line bundle on a compact complex manifold $X$, one defines the volume of $L$ as
\begin{align*}
\mbox{vol}(L):=\limsup_{k\rightarrow +\infty}\frac{n!}{k^n}h^0(X,kL).
\end{align*}
It is well-known that if vol$(L)>0$, then $L$ is big. From \cite{DEL00}, one knows that if vol$(L)>0$, the limsup is in fact a limit, so that the volume vol$(L)$ can be seen as a measure of the bigness of $L$. From the definition, vol$(kL)=k^n\mbox{vol}(L)$. Thus one can also define the volume of a $\mathbb{Q}$-line bundle by setting vol$(L)=k^{-n}\mbox{vol}(kL)$ for some $k$ such that $kL$ is an actual line bundle.

 In \cite{Bou02}, Boucksom introduced a formula expressing the volume of $L$ in terms of $c_1(L)$:
\begin{align*}
\mbox{vol}(L)=\max_T\int_X T^n_{ac}
\end{align*}
for $T$ ranging among the closed positive $(1,1)$-currents in the cohomology class $c_1(L)$, if $L$ is not pseudo-effective, then we let vol$(L)=0$. Where $T_{ac}$ is the absolutely continuous part of the Lebesgue decomposition of $T$  on $X$. Furthermore, the volume of a line bundle is generalized for a cohomology class: for a cohomology class $\alpha\in H^{1,1}(X,\mathbb{R})$, we define
\begin{align*}
\mbox{vol}(\alpha):=\sup_T\int_XT^n_{ac}
\end{align*}
for $T$ ranging over the closed positive $(1,1)$-currents in $\alpha$, in case $\alpha$ is pseudo-effective, otherwise we let vol$(\alpha)=0$. The K\"{a}hler property plays an important role in the proof of  the finiteness  of the above volumes. Here we mention  a couple of results in \cite{Bou02}:
\begin{itemize}
\item [(a)] If $\alpha\in H^{1,1}(X,\mathbb{R})$ is nef, then vol$(\alpha)=\alpha^n$.
\item [(b)] A class $\alpha\in H^{1,1}(X,\mathbb{R})$ is big if and only if vol$(\alpha)>0$.
\end{itemize}
In fact, (b) is a  Grauert-Riemenschneider type criterion for bigness. Let us recall the Grauert-Riemenschneider conjecture (now it is a theorem) for the sake of completeness: a compact complex variety $Y$  is Moishezon if and only if there is a proper non singular modification  $X\rightarrow Y$ and a line bundle $L$ over $X$ such that the curvature is $>0$ on a dense open subset.  A compact complex manifold is said to be Moishezon if it is  birational to a projective manifold. Siu \cite{Siu84} first proved this conjecture by getting  a stronger result that $X$ is Moishezon as soon as $i\Theta_L\geq 0$ everywhere and $i\Theta_L>0$ in at least one point. Later Demailly \cite{Dem85, Dem91} gave another proof of a stronger result than the conjecture by using his holomorphic Morse inequalities. Also Berndtsson \cite{Ber02} gave another proof. It is proved in \cite{JiS93} that a compact complex manifold $X$ is Moishezon, if and only if $X$ admits  an integral K\"{a}hler current, i.e. there exits a big line bundle $L$ on $X$. Now  one can see that (b) is obviously a generalization of Grauert-Rimenschneider  criterion. In fact, it gave a criterion for a transcendental class to be big rather than an integral class.  To conclude, the Philosophy of the study of the  volumes defined above is to ask for the existence of a K\"{a}hler current in a class  $\alpha$ provided that vol$(\alpha)>0$. In \cite{Bou02}, the following conjecture was posed.
\begin{conj}\label{Boucksom conjecture}
If a compact complex manifold $X$ carries a closed positive $(1,1)$-current $T$ with $\int_X T^n_{ac}>0$, then $X$ is in the  \textit{Fujiki} class.
\end{conj}
A compact complex manifold $X$ is said to be in the  Fujiki class if it is bimeromorphic to a K\"{a}hler manifold. Demailly \cite{Dem12} proved that a compact complex manifold $X$ is in the Fujiki class if and only if it carries a K\"{a}hler current.

Throughout this paper, we say that a Hermitian metric $\omega$ satisfies the \textbf{assumption (*)}, if $\partial\overline{\partial}\omega^k=0$ for $k\in \{1,\cdots,n-1\}$.

Now let $(X,\omega)$ be a compact  Hermitian manifold, and $\alpha$ an arbitrary cohomology class $\alpha$ in $H^{1,1}_{\partial\overline{\partial}}(X,\mathbb{R})$, one defines the volume of $\alpha$ as
\begin{align}
\mbox{vol}(\alpha):=\sup_T\int_XT^n_{ac}\notag
\end{align}
for $T$ ranging over the closed positive $(1,1)$-currents in $\alpha$, in case $\alpha$ is pseudo-effective. If it is not, we set $\mbox{vol}(\alpha)=0$.  We will see  that, the supremum involved is always finite under our assumption (*). It is trivial that the volume $\mbox{vol}(\alpha)$ of a big class $\alpha$ is non-zero. Firstly, we will prove that (a) also holds when $(X,\omega)$ is a compact Hermitian manifold endowed with a Gauduchon metric $\omega$ satisfying the assumption (*). Furthermore, by adapting arguments from \cite{Bou02} and  \cite{Chi13}, we are able to prove the following partial solution to Conjecture \ref{Boucksom conjecture}.
\begin{thm}\label{Grauert-Riemenschneider}
Let $X$ be a compact complex manifold, and let  $\omega$ be a Gauduchon metric on $X$ satisfying the  assumption (*). If $X$ carries a pseudo-effective class $\alpha\in H^{1,1}_{\partial\overline{\partial}}(X,\mathbb{R})$ such that vol$(\alpha)>0$, then $X$ is K\"{a}hler.
\end{thm}

Thus for the same reason as in \cite{Bou02}, this  definition  is compatible with the previous one when $X$ is assumed to satisfy the assumption (*).

Since every compact complex surface always  carries  a Gauduchon metric satisfying the assumption (*), Theorem \ref{Grauert-Riemenschneider} states that the Grauert-Riemenschneider type  criterion always holds true on compact complex surface which was proved in \cite{Bou02} by a different argument.

In \cite{Chi13}, Chiose proved that if $X$ is a compact complex manifold, which admits a Gauduchon metric satisfying the assumption (*), a nef class $\alpha\in H^{1,1}_{\partial\overline{\partial}}(X,\mathbb{R})$ has positive volume, then $\alpha$ is a big class and $X$ is in the  \textit{Fujiki} class, and finally K\"{a}hler. The main difference between our Theorem \ref{Grauert-Riemenschneider} and Chiose's result is that we only assume that $\alpha$ is a pseudo-effective class. In general, the nef cone is only a subset of the pseudo-effective cone.

Recently, there has been important progress  on the study of   the structure of compact K\"{a}hler manifolds with nef anticanonical bundles. In \cite{Pau97, Pau98}, it is proved that if $X$ is a compact K\"{a}hler manifold with $K^{-1}_X$ nef, then $\pi_1(X)$ has polynomial growth and, as a consequence it possesses a nilpotent subgroup of finite index. In \cite{Cao13a, Cao13b}, it is proved that for a compact K\"{a}hler manifold $X$ with $K^{-1}_X$ nef, it is projective and rationally connected if and only if $H^0(X, (T^*_X)^{\otimes m})=0$ for all $m\geq 1$. This result is a partial solution to a conjecture attributed to  Mumford.  The following two properties are crucial to prove the above results.
\begin{itemize}
\item[(1)] Let $(X,\omega)$ be a compact K\"{a}hler manifold and $\{\omega\}$ is a K\"{a}hler class on $X$. Then $K^{-1}_X$ is nef if and only if for every $\varepsilon>0$, there exists a K\"{a}hler metric $\omega_\varepsilon=\omega+i\partial\overline{\partial}\phi_\varepsilon$ in the cohomology class $\{\omega\}$ such that Ricci$(\omega_\varepsilon)\geq -\varepsilon\omega_\varepsilon$.
\item[(2)]Let $(X,\omega)$ be a compact K\"{a}hler manifold and  $K^{-1}_X$ is nef. Let \begin{align*}
0=\mathcal{F}_0\subset\mathcal{F}_1\subset\cdots\subset\mathcal{F}_s=T_X
\end{align*}
be a Harder-Narasimhan filtration of $T_X$ with respect to $\omega$. Then
\begin{align*}
\mu_\omega(\mathcal{F}_i/\mathcal{F}_{i-1})\geq 0
\end{align*}
for all $i$.
\end{itemize}
Since $K^{-1}_X$  can be also defined and  there is also an analogue of the Harder-Narasimhan filtration on a compact Hermitian manifold, it is natural to ask whether we can get  similar characterizations of nef $K^{-1}_X$ and the Harder-Narasimhan filtration on a compact Hermitian manifold? In this paper, we get the following
\begin{thm}\label{characterization of nef}
Let $(X,\omega)$ be a compact Hermitian manifold. Then the following properties are equivalent:
\begin{itemize}
\item [(i)] $K_X^{-1}$ is nef.
\item [(ii)] For every $\varepsilon>0$, there exists a smooth  real function $\phi_\varepsilon$, such that $\omega_{\phi_\varepsilon}=\omega+i\partial\overline{\partial}\phi_\varepsilon>0$ and  Ricci$(\omega_{\phi_\varepsilon})\geq -\varepsilon \omega_{\phi_\varepsilon}$.
\end{itemize}
\end{thm}
\begin{thm}\label{semi-positivity}
Let  $(X,\omega)$ be a compact complex manifold with a Gauduchon metric $\omega$ satisfying the assumption (*). Assume that $K_X^{-1}$ is nef. Let
\begin{align*}
0=\mathcal{F}_0\subset\mathcal{F}_1\subset\cdots\subset\mathcal{F}_s=T_X
\end{align*}
be a Harder-Narasimhan filtration of $T_X$ with respect to $\omega$. Then
\begin{align*}
\mu_\omega(\mathcal{F}_i/\mathcal{F}_{i-1})\geq 0
\end{align*}
for all $i$.
\end{thm}

The structure of this paper is as follows. In Section \ref{technical preliminaries}, we prepare the technical preliminaries. In Section \ref{volume of pseudo-effective and nef class}, we  prove that for any nef class $\alpha\in H^{1,1}_{\partial\overline{\partial}}(X,\mathbb{R})$, the volume satisfies $vol(\alpha)=\alpha^n$. It is a generalization of the Theorem 4.1 in \cite{Bou02}. In Section \ref{Grauert-Riemenschneider criterion}, we  prove  Theorem \ref{Grauert-Riemenschneider}. In Section \ref{nef}, we prove Theorem \ref{characterization of nef}. In section \ref{semipositive}, we prove Theorem \ref{semi-positivity}.

\section{Technical preliminaries}\label{technical preliminaries}
Let $X$ be a compact complex $n$-fold. We will use $dd^c$ to denote the operator $\frac{i}{\pi}\partial\overline{\partial}$.
\begin{defin}\label{APC}
A closed real $(1,1)$-current $T$ on $X$ is said to be almost positive if some smooth real $(1,1)$-form $\gamma$ can be found such that $T\geq \gamma$. A function $\varphi\in L^1_{loc}(X)$ is called almost plurisubharmonic if its complex Hessian $dd^c\varphi$ is an almost positive current.

We say that a function $\phi$ on $X$ has analytic singularities along a subscheme $V(\mathscr{I})$ (corresponding to a coherent ideal sheaf $\mathscr{I}$) if there exists $c>0$ such that $\phi$ is locally congruent to $\frac{c}{2}\log(\sum|f_i|^2)$ modulo smooth functions, where $f_1,\cdots, f_r$ are local generators of $\mathscr{I}$.  Note that a function with analytic singularities is automatically almost plurisubharmonic, and is smooth away from the support of $V(\mathscr{I})$.

We say an almost positive $(1,1)$-current  has analytic  singularities, if we can find a smooth form $\theta$ and a function  $\varphi$ on $X$ with analytic   singularities, such that $T=\theta+dd^c\varphi$. Note that one can always write $T=\theta+dd^c\varphi$ with $\theta$ smooth and $\varphi$ almost plurisubharmonic on a compact complex manifold.
\end{defin}

\subsection{$\partial\overline{\partial}$-cohomology}\label{BC-cohomology}
Let $X$ be an arbitrary compact complex manifold of complex dimension $n$. Since the $\partial\overline{\partial}$-lemma does not hold in general, it is better to work with $\partial\overline{\partial}$-cohomology which is defined as
\begin{align}
H^{p,q}_{\partial\overline{\partial}}(X,\mathbb{C})=\big (\mathcal{C}^\infty(X,\Lambda^{p,q}T^*_X)\cap \ker d)/\partial\overline{\partial}\mathcal{C}^\infty(X,\Lambda^{p-1,q-1}T^*_X).\notag
\end{align}

By means of the Fr\"{o}licher spectral sequence, one can see  that  $H^{p,q}_{\partial\overline{\partial}}(X,\mathbb{C})$ is finite dimensional and can be computed either with spaces of smooth forms or with currents. In both cases, the quotient topology of $H^{p,q}_{\partial\overline{\partial}}(X,\mathbb{C})$ induced by the Fr\'{e}chet topology of smooth forms or by the weak topology of currents is Hausdorff, and the quotient map under this Hausdorff topology is continuous and open.

In this paper, we will just need the $(1,1)$-cohomology space $H^{1,1}_{\partial\overline{\partial}}(X,\mathbb{C})$.  The real structure on the space of $(1,1)$-smooth forms (or $(1,1)$-currents) induces a real structure on $H^{1,1}_{\partial\overline{\partial}}(X,\mathbb{C})$, and we denote by $H^{1,1}_{\partial\overline{\partial}}(X,\mathbb{R})$ the space of real points.  A class $\alpha\in H^{1,1}_{\partial\overline{\partial}}(X,\mathbb{C})$ can be seen as an affine space of closed $(1,1)$-currents. We denote by $\{T\}\in H^{1,1}_{\partial\overline{\partial}}(X,\mathbb{C}) $ the class of the current $T$. Since $i\partial\overline{\partial}$ is a real operator (on forms of currents), if $T$ is a real closed $(1,1)$-current, its class $\{T\}$ lies in $H^{1,1}_{\partial\overline{\partial}}(X,\mathbb{R})$ and consists of all the closed currents $T+i\partial\overline{\partial}\varphi$ where $\varphi$ is a real current of degree $0$.

\begin{defin}\label{pf-nef-kahler}
Let $(X,\omega$) be a compact Hermitian manifold. A cohomology class $\alpha\in H^{1,1}_{\partial\overline{\partial}}(X,\mathbb{R})$ is said to be \textbf{pseudo-effective} iff it contains a positive current; $\alpha$ is \textbf{nef} iff, for each $\varepsilon>0$, $\alpha$ contains a smooth form $\theta_\varepsilon$ with $\theta_\varepsilon\geq -\varepsilon\omega$; $\alpha$ is \textbf{big} iff it contains a K\"{a}hler current, i.e. a closed $(1,1)$-current $T$ such that $T\geq\varepsilon\omega$ for $\varepsilon>0$ small enough. Finally, $\alpha$ is a \textbf{K\"{a}hler class} iff it contains a K\"{a}hler form.
\end{defin}
Since any two Hermitian forms $\omega_1$ and $\omega_2$ are commensurable ( i.e. $C^{-1}\omega_2\leq \omega_1\leq C\omega_2$ for some $C>0$), these definitions do not depend on the choice of $\omega$.

\subsection{Lebesgue decomposition of a current}\label{Lebesgue decomposition} In this subsection, we refer to \cite{Bou02, MM07}.
For a measure $\mu$ on a manifold $M$ we denote by $\mu_{ac}$ and $\mu_{sing}$ the uniquely determined absolute continuous and singular measures (with respect to the Lebesgue measure on $M$) such that
\begin{align*}
\mu=\mu_{ac}+\mu_{sing}
\end{align*}
which is called the Lebesgue decomposition of $\mu$. If $T$ is a $(1,1)$-current of order $0$ on $X$, written locally $T=i\sum T_{ij}dz_i\wedge d\overline{z}_j$, we defines its absolute continuous and singular components by
\begin{align*}
T_{ac}&=i\sum (T_{ij})_{ac}dz_i\wedge d\overline{z}_j,\\
T_{sing}&=i\sum (T_{ij})_{sing}dz_i\wedge d\overline{z}_j.
\end{align*}
The Lebesgue decomposition of $T$ is then
\begin{align*}
T=T_{ac}+T_{sing}.
\end{align*}
If $T\geq 0$, it follows that $T_{ac}\geq 0$ and $T_{sing}\geq 0$. Moreover, if $T\geq \alpha$ for a continuous $(1,1)$-form $\alpha$, then $T_{ac}\geq \alpha$, $T_{sing}\geq 0$. The Radon-Nikodym theorem insures that $T_{ac}$ is (the current associated to) a $(1,1)$-form with $L^1_{loc}$ coefficients. The form $T_{ac}(x)^n$ exists for almost all $x\in X$ and is denoted $T^n_{ac}$.

Note that $T_{ac}$ in general is not closed, even when $T$ is, so that the decomposition doesn't induce a significant decomposition at the cohomological level. However, when $T$ is a closed positive $(1,1)$-current with analytic singularities along a subscheme $V$, the residual part $R$ in Siu decomposition (c.f.\cite{Siu74}) of $T$ is nothing but $T_{ac}$, and the divisorial part $\sum_k\nu(T,Y_k)[Y_k]$ is $T_{sing}$. The following facts are well-known.
\begin{lem}[c.f.\cite{Bou02}]\label{push-forward}
Let $f:Y\rightarrow X$ be a proper surjective holomorphic map. If $\alpha$ is a locally integrable form of bidimension $(k,k)$ on $Y$, then the push-forward current $f_*\alpha$ is absolutely continuous, hence a locally integrable form of bidimension $(k,k)$. In particular, when $T$ is a positive current on $Y$, the push-forward current $f_*(T_{ac})$ is absolutely continuous, and we have the formula $f_*(T_{ac})=(f_*T)_{ac}$.

\end{lem}

The absolutely continuous part $T_{ac}$ of a positive current $T$ does not depend continuously on $T$, but we have the following semi-continuity property:

\begin{lem}[c.f.\cite{Bou02}]\label{semi-continuity}
Let $T_k$ be a sequence of positive $(1,1)$-currents converging weakly to $T$. Then one has
\begin{align}
T_{ac}(x)^n\geq \limsup T_{k,ac}(x)^n\notag
\end{align}
 for almost every $x\in X$.
\end{lem}

\subsection{Regularization of currents}\label{regularization of currents}
There are two basic types of regularizations (inside a fixed cohomology class) for closed $(1,1)$-currents, both due to J.-P. Demailly.
\begin{thm}[c.f.\cite{Dem82, Dem92, Bou02}]\label{regularization}
Let $T$ be a closed almost positive $(1,1)$-current on a compact Hermitian manifold $(X,\omega)$. Suppose that $T\geq\gamma$ for some smooth $(1,1)$-form $\gamma$ on $X$. Then
\begin{itemize}
\item [(i)] There exists a sequence of smooth forms $\theta_k$ in $\{T\}$ which converges weakly to $T$, and $\theta_k(x)\rightarrow T_{ac}(x)$ a.e.,  and such that $\theta_k\geq \gamma-C\lambda_k\omega$ where $C>0$ is a constant depending on the curvature of $(T_X,\omega)$ only, and $\lambda_k$ is a decreasing sequence of continuous functions such that $\lambda_k(x)\rightarrow \nu(T,x)$ for every $x\in X$.
\item [(ii)] There exists a sequence $T_k$ of currents with analytic singularities in $\{T\}$ which converges weakly to $T$, and $T_{k,ac}(x)\rightarrow T_{ac}(x)$ a.e., such that $T_k\geq \gamma-\varepsilon_k\omega$ for some sequence $\varepsilon_k>0$ decreasing to $0$, and such that $\nu(T_k,x)$ increases to $\nu(T,x)$ uniformly with respect to $x\in X$.
\end{itemize}
\end{thm}

\subsection{Resolution of singularities}\label{resolution of singularities}
\begin{defin}\label{pull-back}
Let $f:Y\rightarrow X$ be a surjective holomorphic map between compact complex manifolds and $T$ be a closed almost positive $(1,1)$-current on $X$. Write $T=\theta+dd^c\varphi$ for some smooth form $\theta\in\{T\}$, and $\varphi$ an almost plurisubharmonic function on $X$.  We define its  pull back $f^*T$ by $f$ to be $f^*\theta+dd^cf^*\varphi$. Note that this definition is independent of the choices made, and we have $\{f^*T\}=f^*\{T\}$.
\end{defin}

We now use the notations in Definition \ref{APC}. From \cite{Hir64, BiM91, BiM97}, one can  blow-up $X$ along $V(\mathscr{I})$ and resolve the singularities, to get a smooth modification $\mu:\widetilde{X}\rightarrow X$, where $\widetilde{X}$ is a compact complex manifold, such that $\mu^{-1}{\mathscr{I}}$ is just $\mathcal{O}(-D)$ for some simple normal crossing divisor $D$ upstairs.   The pull back $\mu^*T$ clearly has analytic singularities along $V(\mu^{-1}(\mathscr{I}))=D$, thus its Siu decomposition  writes
\begin{align}
\mu^*T=\theta+cD,\notag
\end{align}
where $\theta$ is a smooth $(1,1)$-form.  If $T\geq \gamma$ for some smooth form $\gamma$, then $\mu^*T\geq \gamma$, and thus $\theta\geq \mu^*\gamma$. We call this operation a resolution of the singularities of $T$.

\subsection{Lamari's criterion}\label{lamari}
\begin{thm}
Let $X$ be an $n$-dimensional compact complex manifold and let $\Phi$ be a real $(k,k)$-form, then there exists a real $(k-1,k-1)$-current $\Psi$ such that $\Phi+dd^c\Psi$ is positive iff for any strictly positive $\partial\overline{\partial}$-closed $(n-k,n-k)$-forms $\Upsilon$, we have $\int_X\Phi\wedge\Upsilon\geq 0$.
\end{thm}

\subsection{Gauduchon metrics}\label{Gauduchon metric}
For any $n$-dimensional compact complex manifold $X$, Gauduchon's result \cite{Gau77} tells us there always exists a metric $\omega$ such that $\partial\overline{\partial}\omega^{n-1}=0$. These metrics   are called Gauduchon metrics. Actually, from \cite{Gau84} we know that in the conformal class of every Hermitian metric, there is a  Gauduchon metric.
As a consequence, if the Gauduchon metric $\omega$ satisfies the assumption (*), then for any closed $(1,1)$-current $T$, and $k\in\{1,\cdots,n\}$, the integral $\int_XT^k\wedge\omega^{n-k}$ only depends on the class of $T$ and the metric $\omega$, provided that $T^k$ is well-defined.

The following two  theorems, which we will state without proof,  will play  key roles in this paper.
\begin{thm}[\cite{Che87}]\label{Cherrier}
Let $(X,\omega)$ be a compact Hermitian manifold. The complex Monge-Amp\`{e}re equation
\begin{align}\label{MA2}
(\omega+i\partial\overline{\partial}\phi)^n=e^{\varepsilon\phi-F_\varepsilon}\omega^n
\end{align}
where $\varepsilon>0$ and $F_\varepsilon$ is a smooth function on $X$, has a smooth solution $\phi$ such that $\omega_\phi:=\omega+i\partial\overline{\partial}\phi>0$.
\end{thm}

\begin{thm}[\cite{TW10}]\label{Monge-Ampere equation}
Let $(X,\omega)$ be a compact Hermitian manifold. For any smooth real-valued function $F$ on $X$, there exist a unique real number $C>0$ and a unique smooth real-valued function $\phi$ on $X$ solving
\begin{align}
(\omega+i\partial\overline{\partial}\phi)^n=Ce^F\omega^n,\notag
\end{align}
with $\omega+i\partial\overline{\partial}\phi>0$ and $\sup_X\phi=0$.  Furthermore, if $\partial\overline{\partial}\omega^k=0$ for $1\leq k\leq n-1$, then we have
\begin{align}
C=\frac{\int_X\omega^n}{\int_Xe^F\omega^n}.\notag
\end{align}
\end{thm}
\begin{rem}
The assumption (*) is also used in \cite{GuL10} to solve the complex Monge-Amp\`{e}re equation.
\end{rem}
\begin{rem}

 One should be careful that if $\omega$ is Gauduchon, and $\phi $ is a smooth function on $X$ such that  $\omega_\phi:=\omega+i\partial\overline{\partial}\phi>0$, then $\omega_\phi$ is not Gauduchon in general.

\end{rem}
\begin{lem}\label{assumption lemma}
Suppose $\omega$ is a Hermitian form on $X$ satisfying  the assumption (*). Then for any smooth function $\phi$ on $M$, $\omega_\phi$ also satisfies the assumption (*).
\end{lem}
\begin{proof}
It is  a direct and easy computation.
\end{proof}

\subsection{Finiteness of the volume}
The following two lemmas are  small generalizations of those in \cite{Bou02}. The proof is similar with small modifications, we give the proof here for the sake of completeness.
\begin{lem}\label{Bouksom 1}
Let $T$ be any closed  $(1,1)$-current on a compact Hermitian manifold $(X,\omega)$ with $T\geq \gamma$, where $\omega$ is the Gauduchon metric satisfying the assumption (*) and $\gamma$ is a continuous  $(1,1)$-form on $X$. Then one can define the Lelong number $\nu(T,x)$ for $T$ at $x$ to be $\nu(T+\beta,x)$, where $\beta$ is a smooth closed $(1,1)$-form near $x$ such that $T+\beta\geq 0$ and   $\nu(T,x)$  can be bounded by a constant depending only on the $\partial\overline{\partial}$-cohomology class of $T$.
\end{lem}
\begin{proof}
By definition $\nu(T+\beta,x) $ is (up to a constant  depending on $\omega$ near $x$) the limit for $r\rightarrow 0_+$ of
\begin{align*}
\nu(T+\beta,x,r):=\frac{(n-1)!}{(\pi r^2)^{n-1}}\int_{B(x,r)}(T+\beta)\wedge\omega^{n-1},
\end{align*}
which is known to be an increasing function of $r$. Since $\beta$ is smooth, one can see that the limit is independent of $\beta$, which means that the definition of $\nu(T,x)$ is well-defined. Choose a constant $C$ such that $C\omega\geq -\gamma$, then $T+C\omega\geq 0$ on $X$. Thus if we choose $r_0$ small enough to ensure that each ball $B(x,r_0)$ is contained in a coordinate chart, we get $\nu(T,x)\leq \nu(T+C\omega,x,r_0)\leq \int_X(T+C\omega)\wedge \omega^{n-1}=\int_XT\wedge\omega^{n-1}+C\int_X\omega^n$. But the last term is  a quantity depending on the cohomology class $\{T\}$ since $\omega$ satisfies the assumption (*).
\end{proof}

\begin{lem}\label{Boucksom 2}
Under the same assumption as in Lemma \ref{Bouksom 1},
one sees that the integrals $\int_XT^k_{ac}\wedge\omega^{n-k}$ are finite for each $k=0,\cdots,n$ and can  be bounded in terms of $\omega$ and the $\partial\overline{\partial}$-cohomology class of $T$ only.
\end{lem}
\begin{proof}Since $\gamma$ is continuous and $X$ is compact, there exists a constant $C>0$, such that $T\geq -C\omega$, and thus $T_{ac}\geq -C\omega$. Let $-C\leq\lambda_1\leq\cdots\leq\lambda_n$ be the eigenvalues of $T_{ac}$. There is  a simple observation:
whenever $\lambda_k$ is negative or positive,  $|\lambda_k|\leq \lambda_k+2C$ always holds. Thus
\begin{align*}
\big|\int_XT^k_{ac}\wedge\omega^{n-k}\big|\leq \int_X(T_{ac}+2C\omega)^k\wedge\omega^{n-k}.
\end{align*}
It suffices to prove the right hand side of above inequality is uniformly bounded.  Choose a sequence $T_k$ of smooth forms approximating $T$ as in Theroem \ref{regularization}.  Since $T_k\geq -C\omega-C\lambda_k\omega=-C(1+\lambda_k)\omega$ for some constant $C>0$ depending only on $(X,\omega)$ only and continuous functions $\lambda_k(x)$ decreasing to $\nu(T,x)$, we find using Lemma \ref{Bouksom 1} a constant also denoted by $C$ and depending on $(X,\omega)$ and the cohomology class $\{T\}$ only such that $T_k+C\omega\geq 0$. But now
\begin{align*}
\int_X(T_k+C\omega)^l\wedge\omega^{n-l}=\int_X(\{T\}+C\omega)^l\omega^{n-l}
\end{align*}
does not depend on $k$ since our $\omega$ satisfies the assumption (*), so the result follows by Fatou's lemma, since $T_k+C\omega$ is a smooth form converging to $T_{ac}+C\omega$ a.e.
\end{proof}

\subsection{Harder-Narasimhan filtration on compact Hermitian manifold}
In this subsection, we refer  to Bruasse \cite{Bru01}.
Let $(X,\omega)$ be a compact Hermitian manifold endowed with a Gauduchon metric $\omega$. Let $L$ be a holomorphic line bundle on $X$ and $h$ be a Hermitian metric on $L$. Let $\Theta_{L,h}$ be the Chern curvature form of $L$ associated to $h$. Since it is independent of $h$ up to a $\partial\overline{\partial}$-exact term and $\omega$ is Gauduchon, the $\omega$-degree of $L$ given by
\begin{align*}
\deg_\omega(L)=\int_X\Theta_{L,h}\wedge \omega^{n-1}
\end{align*}
is a well-defined real number independent of $h$.

Now if $\mathcal{F}$ is a rank $p$ coherent sheaf of $\mathcal{O}_X$-modules, consider the holomorphic line bundle $\det\mathcal{F}=(\wedge^p\mathcal{F})^{**}$. Then we have
\begin{defin}
\begin{itemize}
\item [(i)]  The $\omega$-degree of $\mathcal{F}$ is
\begin{align*}
\deg_\omega(\mathcal{F}):=\deg_\omega(\det \mathcal{F}).
\end{align*}
\item [(ii)] If $\mathcal{F}$ is nontrivial and torsion-free, then we define its slope (or $\omega$-slope) by
    \begin{align*}
    \mu(\mathcal{F}):=\frac{\deg_\omega(\mathcal{F})}{\mbox{rank}(\mathcal{F})}.
    \end{align*}

\end{itemize}
\end{defin}

\begin{defin} A torsion-free coherent sheaf $\mathcal{E}$ is called $\omega$-(semi) stable if for every coherent subsheaf $\mathcal{F}\subset \mathcal{E}$ with $0<\mbox{rank} \mathcal{F}<\mbox{rank}\mathcal{E}$, one has $\mu(\mathcal{F})<(\leq )\mu(\mathcal{E})$.
\end{defin}

\begin{defin}
Let $(X,\omega)$ be a compact complex manifold of dimension $n$ endowed with a Gauduchon metric $\omega$. Let $\mathcal{F}$ be a torsion free coherent sheaf over $X$. A Harder-Narasimhan filtration for $\mathcal{F}$ is a flag:
\begin{align*}
0=\mathcal{F}_0\subset \mathcal{F}_1\subset \cdots\subset \mathcal{F}_{s-1}\subset \mathcal{F}_s=\mathcal{F}
\end{align*}
of subsheaves of $\mathcal{F}$ with the following two properties:
\begin{itemize}
\item[(1)] $\mathcal{F}_i/\mathcal{F}_{i-1}$ is $\omega$-semi-stable for $1\leq i\leq s-1$,
\item[(2)]$\mu(\mathcal{F}_{j+1}/\mathcal{F}_j)<\mu(\mathcal{F}_j/\mathcal{F}_{j-1})$ for $1\leq j\leq s-1$.
\end{itemize}
In fact, $\mathcal{F}_i/\mathcal{F}_{i-1}$ is the maximal $\omega$-semi-stable subsheaf of $\mathcal{F}/\mathcal{F}_{i-1}$ for $1\leq i\leq s-1$.
 \end{defin}

\begin{thm}[\cite{Bru01}]
Let $(X,\omega)$ be a compact complex manifold of dimension $n$ endowed with a Gauduchon metric $\omega$. Let $E$ be a holomorphic vector bundle of rank $r$ over $X$. It possesses a unique Harder-Narasimhan filtration.
\end{thm}

\section{Volume of a nef class}\label{volume of pseudo-effective and nef class}
\begin{thm}\label{volume of nef}Let $(X,\omega)$ be a compact Hermitian manifold endowed with a Gauduchon metric $\omega$ satisfying the assumption (*). If $\alpha\in H^{1,1}_{\partial\overline{\partial}}(X,\mathbb{R})$ is a  nef class, then one has $\mbox{vol}(\alpha)=\alpha^n$.
\end{thm}
\begin{proof}The proof is a small modification of that in \cite{Bou02}, for the sake of completeness we give the proof here. Firstly, we prove that  for every positive $T\in \alpha$, we have $\int_XT^n_{ac}\leq \alpha^n$, which will certainly imply $\mbox{vol}(\alpha)\leq \alpha^n$.  Write $T=\theta+dd^c\varphi$ with $\theta$ a smooth form. We consider a sequence $T^{(1)}_k=\theta+dd^c\varphi^{(1)}_k$ of smooth forms given by (i) of Theorem \ref{regularization}. Since $\alpha$ is nef, by definition, there exists a sequence of smooth functions $\varphi^{(2)}_k$ and a sequence of positive numbers $\varepsilon_k\rightarrow 0$, such that $T^{(2)}_k=\theta+dd^c\varphi^{(2)}_k$ satisfies $T^{(2)}_k\geq -\varepsilon_k\omega$. Set $\varphi^{(3)}_k:=\max_\eta(\varphi^{(2)}_k-C_k,\varphi^{(1)}_{j_k})$, where $C_k\rightarrow +\infty$ as $k\rightarrow +\infty$ and $\varphi_{j_k}^{(1)}$ is a properly chosen subsequence of $\varphi^{(1)}_k$ (c.f. \cite[Page 1050]{Bou02}), then $\varphi^{(3)}_k$ is a smooth function, and it is proved in \cite[Lemma 4.2]{Bou02} that, $T^{(3)}_k:=\theta+dd^c\varphi^{(3)}_k$ is a smooth form such that $T^{(3)}_k(x)\rightarrow T_{ac}(x)$ a.e., and $T^{(3)}_k\geq -\delta_k\omega$ for some sequence $\delta_k>0$ converging to $0$. Since $T^{(3)}_k+\delta_k\omega$ also converges to $T_{ac}$ a.e., Fatou's lemma gives us
\begin{align}
\int_XT^n_{ac}\leq \liminf\limits_{k\rightarrow \infty}\int_X(T^{(3)}_k+\delta_k\omega)^n,\notag
\end{align}
and the latter integral depends only on the class $\alpha$ and $\delta_k$, thus it converges to $\alpha^n$. That is $\int_XT^n_{ac}\leq \alpha^n$.

Secondly, we want to show $\mbox{vol}(\alpha)\geq \int\alpha^n$. Normalize our Gauduchon metric $\omega$ in the assumption (*), such that $\int_X\omega^n=1$. For $\varepsilon>0$, there exists a closed form $T\in \alpha$ such that $T+\varepsilon\omega>0$. Using Theorem \ref{Monge-Ampere equation}, one can solve the  equation
\begin{align}
\tau^n_\varepsilon=\big(\int_X(T+\varepsilon\omega)^n\big)\omega^n,\notag
\end{align}
where $\tau_\varepsilon=T+\varepsilon\omega+dd^c\varphi_\varepsilon>0$, and $\phi_\varepsilon$ is normalized so that   $\sup_X\varphi_\varepsilon=0$.  Since the family $\tau_\varepsilon-\varepsilon\omega\in\alpha$ represents a bounded set of cohomology classes, it is bounded in mass and we can thus extract some weak limit $T=\lim\limits_{\varepsilon\rightarrow 0}(\tau_\varepsilon-\varepsilon\omega)=\lim\limits_{\varepsilon\rightarrow 0}\tau_\varepsilon$, where the second equality holds  because $\lim\limits_{\varepsilon\rightarrow 0}\varepsilon\omega=0$ in the strong sense. By Lemma \ref{semi-continuity}, we get $T^n_{ac}\geq (\int\alpha^n)\omega^n$, by integrating, we get $\mbox{vol}(\alpha)\geq \alpha^n$.
\end{proof}



\section{The Grauert-Riemenschneider criterion: Proof of Theorem \ref{Grauert-Riemenschneider} }\label{Grauert-Riemenschneider criterion}

\begin{proof}[Proof of Theorem \ref{Grauert-Riemenschneider}]
From the definition of $\mbox{vol}(\alpha)$, one can find a positive closed current $S\in \alpha$ such that $\int_XS^n_{ac}>\frac{\mbox{vol}(\alpha)}{2}>0$.  Apply (ii) in Theorem \ref{regularization}, combined with Fatou's lemma, we can construct a sequence $T_k$ of closed currents with analytic singularities in $\alpha$ such that $T_k\geq -\varepsilon_k\omega$ and $\int_XT_{k,ac}^n\geq c$ for some uniform lower bound $c>0$. Where $\varepsilon_k\rightarrow 0$ as $k\rightarrow \infty$. In fact, one has $0<\int_XS^n_{ac}\leq \liminf_k\int_X(T_{k,ac}+\varepsilon_k\omega)^n$, thus one can subtract a subsequence which is  denoted by $T_{k,ac}+\varepsilon_k\omega$ such that $\int_X(T_{k,ac}+\varepsilon_k\omega)^n>C$ for some uniform constant $C$. But
\begin{align*}
\int_X(T_{k,ac}+\varepsilon_k\omega)^n=\int_XT^n_{k,ac}+\sum_{l=1}^n\varepsilon^l_k\binom{n}{l}\int_XT_{k,ac}^{n-l}\wedge\omega^l
\end{align*}
  where the second term is uniformly bounded for $k$ large enough (say $0<\varepsilon_k<<1$) by Lemma \ref{Boucksom 2}.
  For each $k$, we  choose a  smooth proper modification $\mu_k:X_k\rightarrow X$ such that $\mu_k^*T_k=\theta_k+D_k$, with $\theta_k$ ($\geq -\varepsilon_k\mu_k^*\omega$) a smooth closed form and $E_k$ an real effective divisor.
Set $\Omega_k=\varepsilon_k\mu^*_k\omega$. It is easy to see that $0<c\leq \int_XT^n_{k,ac}=\int_{X_k}(\mu^*_kT_k)^n_{ac}=\int_{X_k}\theta_k^n$. Select on each $X_k$ a Gauduchon metric  $\widetilde{\omega}_k$ which also satisfies the assumption (*) by the following

\begin{lem}[c.f. \cite{Dem12}]\label{blow-up}Suppose that $(X,\omega)$ is a compact complex manifold satisfying the assumption (*). Let $\mu:\widetilde{X}\rightarrow X$ is a smooth modification (a tower of blow-ups).  Then  there exists a Gauduchon metric $\Omega$ satisfying the assumption (*) on $\widetilde{X}$.
\end{lem}
\begin{proof}
Suppose that $\widetilde{X}$ is obtained as a tower of blow-ups
\begin{align}
\widetilde{X}=X_N\rightarrow X_{N-1}\rightarrow \cdots\rightarrow X_1\rightarrow X_0=X,
\end{align}
where $X_{j+1}$ is the blow-up of $X_j$ along a smooth center $Y_j\subset X_j$. Denote by $E_{j+1}\subset X_{j+1}$ the exceptional divisor, and let $\mu_j:X_{j+1}\rightarrow X_j$ be the blow-up map. The line bundle $\mathcal{O}(-E_{j+1})|_{E_{j+1}}$ is equal to $\mathcal{O}_{P(N_j)}(1)$ where $N_j$ is the normal bundle to $Y_j$ in $X_j$. Pick an arbitrary smooth Hermitian metric on $N_j$, use this metric to get an induced Fubini-Study metric on $\mathcal{O}_{P(N_j)}(1)$, and finally extend this metric as a smooth Hermitian metric on the line bundle $\mathcal{O}(-E_{j+1})$. Such a metric has positive curvature along tangent vectors of $X_{j+1}$ which are tangent to the fibers of $E_{j+1}=P(N_j)\rightarrow Y_j$. Assume further that $\omega_j$ is a Gauduchon metric satisfying assumption (*) on $X_j$. Then
\begin{align}
\Omega_{j+1}=\mu^*_j\omega_j-\varepsilon_{j+1}u_{j+1}
\end{align}
where $\mu^*_j\omega_j$ is  semi-positive on $X_{j+1}$, positive definite on $X_{j+1}\setminus E_{j+1}$, and also positive definite on tangent vectors of $T_{X_{j+1}}|_{E_{j+1}}$ which are not tangent to the fibers of $E_{j+1}\rightarrow Y_j$. It  is then easily to see that $\Omega_{j+1}>0$   by taking $\varepsilon_{j+1}\ll 1$. Thus our final candidate $\Omega$ on $\widetilde{X}$ has the form $\Omega=\mu^*\omega-\sum\varepsilon_j\widetilde{u}_j$, where $\widetilde{u}_j=(\mu_{N-1}\circ \cdots\circ \mu_{j})^*u_j$. Since every $u_j$ is a curvature term of a line bundle, the term $\sum\varepsilon_j\widetilde{u}_j$ is $d$-closed. Now from Lemma \ref{assumption lemma}  our $\Omega$ satisfies the assumption (*).
\end{proof}
Now we want to show that the class $\{\theta_k\}$ is big for $k$ large.  It suffices to show that there exists $\varepsilon_0>0$ and a distribution $\chi$ such that $\theta_k+dd^c\chi\geq \varepsilon_0 \widetilde{\omega}_k$. According to Lamari's criterion \cite{Lam99}, c.f. Section \ref{lamari}, this is equivalent to showing that
\begin{align}
\int_X\theta_k\wedge g^{n-1}\geq \varepsilon_0\int_X\widetilde{\omega}_k\wedge g^{n-1}\notag
\end{align}
for any Gauduchon metric $g$ on $X_k$. Here we use a theorem of Michelsohn \cite{Mic83} which states that every strictly positive $(n-1,n-1)$-form $\beta$ has a $(1,1)$ root  $g$ such that $\beta=g^{n-1}$. Suppose to the contrary that for any $m\in \mathbb{N}$, there exists $\omega_m$ a Gauduchon metric on $X_k$ such that
\begin{align}
\int_{X_k}\theta_k\wedge\omega_m^{n-1}\leq \frac{1}{m}\int_{X_k}\widetilde{\omega}_k\wedge\omega_m^{n-1}.\notag
\end{align}
We can assume that
\begin{align}
\int_{X_k}\widetilde{\omega}_k\wedge\omega^{n-1}_m=1\notag
\end{align}
and therefore
\begin{align}
\int_{X_k}\theta_k\wedge \omega^{n-1}_m\leq \frac{1}{m}.\notag
\end{align}

From Theorem \ref{Monge-Ampere equation}, we can solve the equation
\begin{align}\label{solution}
(\theta_k+\Omega_k+\frac{1}{m}\widetilde{\omega}_k+dd^c\varphi_m)^n=C_m\omega_m^{n-1}\wedge \widetilde{\omega}_k
\end{align}
for a function $\varphi_m\in\mathcal{C}^\infty(X_k,\mathbb{R}) $ such that if we set $$\alpha_m=\theta_k+\Omega_k+\frac{1}{m}\widetilde{\omega}_k+dd^c\varphi_m,$$ then $\alpha_m>0$. The constant $C_m$ is given by
\begin{align}\label{estimate of Cm}
C_m&=\int_{X_k}\big(\theta_k+\Omega_k+\frac{1}{m}\widetilde{\omega}_k\big)^n\geq \int_{X_k}(\theta+\Omega_k)^n\\
&=\int_{X_k}\theta^n_k+O(\varepsilon_k)\int_X\sum_{p\geq1}T^p_{k,ac}\wedge\omega^{n-p}\notag\\
&\geq C>0.\notag
\end{align}
The third inequality follows from Lemma \ref{Boucksom 2} and for $k$ sufficiently large. Where $C$ is a uniform constant depends only on the cohomology class $\alpha$ and $\omega$. Now
\begin{align}
\int_{X_k}\alpha_m^{n-1}\wedge\widetilde{\omega}_k=\int_{X_k}\widetilde{\omega}_k\wedge(\theta_k+\frac{1}{m}\widetilde{\omega}_k+\Omega_k)^{n-1}\leq \int_{X_k}\widetilde{\omega}_k\wedge(\theta_k+\widetilde{\omega}_k)^{n-1}\notag
\end{align}
But one have
\begin{align}
\int_{X_k}\widetilde{\omega}_k\wedge(\theta_k+\widetilde{\omega}_k)^{n-1}\leq C'\int_X\omega\wedge(T_{k,ac}+\omega)^{n-1}\leq M.\notag
\end{align}
The second inequality holds since $\widetilde{\omega}_k=\mu_k^*\omega+\sum \varepsilon_j\widetilde{u}_j$ and $\varepsilon_j$ can be chosen sufficiently small which can be easily seen from  Lemma \ref{blow-up} and similar argument as in (\ref{estimate of Cm}), $M$ is also a uniform constant which depends only on the cohomology class $\alpha$ and the metric $\omega$ on $X$, and the  last inequality is  due to  Lemma \ref{Boucksom 2}.

Set
\begin{align}
E=\Big\{\frac{\alpha_m^{n-1}\wedge \widetilde{\omega}_k}{\omega_m^{n-1}\wedge\widetilde{\omega}_k}>2M\Big\}\notag
\end{align}
then
\begin{align}\label{chio1}
\int_E\omega_m^{n-1}\wedge \widetilde{\omega}_k\leq \frac{1}{2}.
\end{align}
Therefore on $X_k\setminus E$, we have $\alpha_m^{n-1}\wedge \widetilde{\omega}_k\leq 2M\omega_m^{n-1}\wedge \widetilde{\omega}_k$. By looking at the eigenvalues of $\alpha_m$ with respect to $\omega$, from (\ref{solution}), it follows that on $X_k\setminus E$, we have
\begin{align}
\alpha_m\geq \frac{C_m}{2nM}\widetilde{\omega}_k.\notag
\end{align}
Therefore
\begin{align}\label{chio2}
\int_{X_k}\alpha_m\wedge\omega^{n-1}_m\geq \int_{X_k\setminus E}\alpha_m\wedge\omega^{n-1}_m\geq \frac{C_m}{2nM}\int_{X_k\setminus E}\widetilde{\omega}_k\wedge \omega^{n-1}_m=\\
=\frac{C_m}{2nM}\big(\int_{X_k}\widetilde{\omega}_k\wedge\omega_m^{n-1}-\int_E\widetilde{\omega}_k\wedge\omega^{n-1}_m\big)\geq \frac{C}{4nM}.\notag
\end{align}
On the other hand,
\begin{align}\label{chio3}
\int_{X_k}\alpha_m\wedge\omega^{n-1}_m\leq\int_{X_k}\theta_k\wedge\omega^{n-1}_m+\frac{2}{m}\int_{X_k}\widetilde{\omega}_k\wedge\omega_m^{n-1}\leq \frac{3}{m},
\end{align}
which is  a contradiction for $m>>0$. Here the first inequality in (\ref{chio3}) holds for $k$ sufficiently large such that $\Omega_k=\varepsilon_k\mu^*_k\omega\leq \frac{1}{m}\widetilde{\omega}_k$. Therefore $\theta_k$ is big. i.e. there exists a K\"{a}hler current $\Theta\in \{\theta_k\}$ on $X_k$, hence a K\"{a}hler current $(\mu_{k})_*(\Theta+D_k)$ (see Lemma \ref{push-forward}) on $X$. Thus it follows that $X$ is in the \textit{Fujiki} class. Theorem 2.2 in \cite{Chi} implies that a manifold in the \textit{Fujiki} class and which is strong K\"{a}hler with torsion (i.e. it supports a $\partial\overline{\partial}$-closed Hermitian metric), is in fact K\"{a}hler.

\end{proof}

\begin{rem}
It is a fact that in general the pseudo-effective cone (even the big cone) and the nef cone on a compact complex manifold are not the same. In general, the nef cone is contained in the pseudo-effective cone. But the converse is not true. For example, the  exceptional divisor of a blowing-up along one point in $\mathbb{C}P^2$ is pseudo-effective but not nef.  To characterize the pseudo-effective cone is an important question in the study of   complex geometry.
\end{rem}

\begin{rem}Recently, Tosatti \cite{Tos15} gave a proof of (b) using ideas very close to our proof of Theorem \ref{Grauert-Riemenschneider}.
\end{rem}

\begin{rem} In Demailly's book \cite{Dem12}, Demailly introduced the following  definition of vol$_D(\alpha)$ for a pseudo-effective class $\alpha$ on K\"{a}hler manifolds.
 \begin{defin}[c.f. \cite{Dem12}]
Let $X$ be a compact K\"{a}hler manifold. The volume, or mobile self-intersection of a class $\alpha\in H^{1,1}(X,\mathbb{R})$ is defined  to be
\begin{align*}
\mbox{vol}_D(\alpha)=\sup_{T\in \alpha}\int_{X\setminus \mbox{sing}(T)}T^n=\sup_{T\in \alpha}\int_{\widetilde{X}}\beta^n>0
\end{align*}
where the supremum is taken over all K\"{a}hler currents $T\in \alpha$ with logarithmic poles, and $\mu^*T=[E]+\beta$ with respect to some modification $\mu:\widetilde{X}\rightarrow X$. Correspondingly, we set $\mbox{vol}(\alpha)=0$ if $\alpha\not\in \mathcal{E}^o$.
\end{defin}

 It is almost trivial that vol$_D(\alpha)\leq \mbox{vol}(\alpha)$. From Theorem \ref{Grauert-Riemenschneider}, one can now see that if the compact Hermitian manifold $(X,\omega)$ satisfies the assumption (*) and  vol$(\alpha)>0$, then $\alpha$ is big and $X$ is K\"{a}hler. Thus it is natural to ask whether vol$_D(\alpha)=\mbox{vol}(\alpha)$ on a K\"{a}hler manifold.

Firstly, it is easy to see that $\mbox{vol}_D(\alpha)=\sup_{T\in \alpha}\int_{X}T_{ac}^n$, where the supremum is taken over all K\"{a}hler currents $T\in \alpha$ with logarithmic poles. In particular, one has $\mbox{vol}_D(\alpha)\leq \mbox{vol}(\alpha)$.

Secondly, one have $\mbox{vol}(\alpha)\leq \mbox{vol}_D(\alpha)$.
In fact, it is trivial that $\mbox{vol}(\alpha)=0\Leftrightarrow\mbox{vol}_D(\alpha)=0$. Otherwise it is a direct consequence of the following version of Fujita's theorem due to Boucksom.
\begin{thm}[c.f. \cite{Bou02}]Let $X$ be a compact K\"{a}hler manifold, and let $\alpha\in H^{1,1}(X,\mathbb{R})$ is a big class on $X$. Then for every $\varepsilon>0$, there exists a modification $\mu:\widetilde{X}\rightarrow X$, a K\"{a}hler class $\omega$ and an effective real divisor $D$ on $\widetilde{X}$ such that
\begin{itemize}
\item  $\mu^*\alpha=\omega+\{D\}$ as cohomology classes,
\item  $|\mbox{vol}(\alpha)-\mbox{vol}(\omega)|<\varepsilon$.
\end{itemize}
\end{thm}

To conclude, we prove the following
\begin{prop}
Let $X$ be a compact K\"{a}hler manifold, for any $\alpha$ in $H^{1,1}(X,\mathbb{R})$,  $\mbox{vol}(\alpha)=\mbox{vol}_D(\alpha)$.
\end{prop}

\end{rem}

\section{A characterization of the nef anti-canonical bundle on $X$: Proof of Theorem \ref{characterization of nef}}\label{nef}

Let $(X,\omega)$ be a compact complex manifold with a Hermitian  metric $\omega$. Denote by Ricci$(\omega)$ the Chern Ricci curvature of $(X,\omega)$, i.e. the Chern curvature of $K_X^{-1}$  corresponding to the  Hermitian metric induced by the Hermitian metric $\omega$ of $X$.

Firstly, we prove that (i) implies (ii). Suppose $L:=K_X^{-1}$ is nef, that is for any $\varepsilon>0$, there is a smooth hermitian metric $h_\varepsilon$ of $L$, such that $\Theta_{L,h_\varepsilon}\geq-\varepsilon\omega$. From the fact that $\Theta_{L,h_\varepsilon}$ is the Chern Ricci curvature of $X$ and thus a representative of the first Chern class of $X$. One asks for a $\phi$, such that $\omega_\varepsilon:=\omega+i\partial\overline{\partial}\phi>0$ and Ricci$(\omega_\varepsilon)\geq -\varepsilon\omega_\varepsilon$. Let us find out what equation should such $\phi$ satisfy. Now let $u_\varepsilon:=\Theta_{L,h_\varepsilon}\geq-\varepsilon\omega$. Then $u_\varepsilon=\mbox{Ricci}(\omega)+i\partial\overline{\partial}F_\varepsilon$. It thus suffices to  find a $\phi$ such that
\begin{align}\label{equivalent equation1}
\mbox{Ricci}(\omega_\varepsilon)=-\varepsilon\omega_\varepsilon+\varepsilon\omega+u_\varepsilon,
\end{align}
which is equivalent to  equation (\ref{MA2}). In fact,
\begin{align*}
i\partial\overline{\partial}\log\omega_\varepsilon^n-i\partial\overline{\partial}\log\omega^n&=\mbox{Ricci}(\omega)-\mbox{Ricci}(\omega_\varepsilon)\\
&=\varepsilon(\omega_\varepsilon-\omega)+\mbox{Ricci}(\omega)-u_\varepsilon\\
&=i\partial\overline{\partial}(\varepsilon\phi-F_\varepsilon).
\end{align*}
From Theorem \ref{Cherrier}, one concludes that for any $\varepsilon>0$, there is a smooth function $\phi_\varepsilon$, such that $\omega_\varepsilon:=\omega+i\partial\overline{\partial}\phi>0$ and $\mbox{Ricci}(\omega_\varepsilon)\geq-\varepsilon\omega_\varepsilon$.

Conversely, since $\mbox{Ricci}(\omega_\varepsilon)\geq-\varepsilon\omega_\varepsilon$ and $\omega_\varepsilon:=\omega+i\partial\overline{\partial}\phi>0$, one can easily conclude that $\mbox{Ricci}(\omega_\varepsilon)+i\partial\overline{\partial}(\varepsilon\phi)\geq-\varepsilon\omega$. But $\mbox{Ricci}(\omega_\varepsilon)+i\partial\overline{\partial}(\varepsilon\phi)$ is precisely a curvature form associated to a  Hermitian metric on $K_X^{-1}$. Thus one gets the nefness of $K_X^{-1}$.

\section{A semi-positive property of the Harder-Narasimhan filtration of $T_X$: Proof of Theorem \ref{semi-positivity}}\label{semipositive}

The  proof given here is along the line of the proof in Demailly \cite{Dem14}.
First consider the case where the filtration is regular, i.e. all sheaves $\mathcal{F}_i$ and their quotients $\mathcal{F}_i/\mathcal{F}_{i-1}$ are vector bundles. By the stability condition, it is sufficient to prove that
\begin{align*}
\int_Xc_1(T_X/\mathcal{F}_i)\wedge\omega^{n-1}\geq 0
\end{align*}
for all $i$.
From Theorem \ref{characterization of nef}, for each $\varepsilon>0$, there is a smooth function $\phi_\varepsilon$ such that $\omega_\varepsilon:=\omega+i\partial\overline{\partial}\phi_\varepsilon>0$, and Ricci$(\omega_\varepsilon)\geq-\varepsilon\omega_\varepsilon$. This is equivalent to the pointwise estimate
\begin{align*}
i\Theta_{T_X,\omega_\varepsilon}\wedge\omega_{\varepsilon}^{n-1}\geq-\varepsilon\cdot\mbox{Id}_{T_X}\omega^n_\varepsilon.
\end{align*}
Taking the induced metric on $T_X/\mathcal{F}_i$ (which we also denot by $\omega_\varepsilon$), the second fundamental form contributes  nonnegative terms on the quotient, hence the $\omega_\varepsilon$-trace yields
\begin{align*}
\mbox{Trace}(i\Theta_{T_X/\mathcal{F}_i,\omega_\varepsilon}\wedge\omega_{\varepsilon}^{n-1})\geq-\varepsilon\mbox{rank}(T_X/\mathcal{F}_i)\omega^n_\varepsilon.
\end{align*}
Therefore, putting $r_i=\mbox{rank}(T_X/\mathcal{F}_i)$, since for a line bundle, the curvature differs by $\partial\overline{\partial}$-exact terms from different choices of the Hermitian metric and both $\omega$ and $\omega_\varepsilon$ satisfy the assumption (*), we get
\begin{align*}
\int_Xc_1(T_X/\mathcal{F}_i)\wedge\omega^{n-1}&=\int_Xc_1(T_X/\mathcal{F}_i)\wedge\omega_\varepsilon^{n-1}\\
&\geq -\varepsilon r_i\int_X\omega_\varepsilon^n=-\varepsilon r_i\int_X\omega^n,
\end{align*}
and we are done. In case there are singularities, from the construction \cite{Bru01} they occur only on some analytic subset $S\subset X$ of codimension $2$. The first Chern forms calculated on $X\setminus S$ extend as locally integrable currents on $X$ and do not contribute any mass on $S$. The above calculations are still valid. Thus we complete the proof of Theorem \ref{semi-positivity}.

\subsection*{Acknowledgements}
The author would like to thank Professor Jean-Pierre Demailly for his talk about structures of K\"{a}hler manifolds given at Peking University, pointing out to him the key observations in the proof of Lemma \ref{Boucksom 2}, and many inspiring discussions from which the author benefited a lot. The author is grateful to  Professor Xiangyu Zhou and Kefeng Liu for their constant support,  encouragement, and  helpful discussions, to Professor Valentino Tosatti for many helpful discussions. Part of this research was done at the Department of Mathematics in University of California at Los Angeles under the support of Chinese Scholarship Council.

\end{document}